\documentclass[letterpaper,11pt]{article}

\usepackage{color}
\usepackage{grffile}
\usepackage{ctable}
\usepackage{amssymb}
\usepackage{amsmath}
\usepackage[dvips]{epsfig}
\usepackage[small]{caption}
\usepackage{graphicx}
\usepackage[letterpaper,text={16cm,23cm},centering]{geometry}

\usepackage{tikz}
\usepackage{txfonts}
\usetikzlibrary{matrix,arrows,decorations.pathmorphing}

\newtheorem{lem}{Lemma}[section]
\newtheorem{thm}{Theorem}
\newtheorem{prop}[lem]{Proposition}
\newtheorem{cor}[lem]{Corollary}
\newtheorem{rem}[lem]{Remark}

\newenvironment{proof}%
 {\begin{trivlist} \item[]{\bf Proof. }}%
 {\hspace*{\fill}$\rule{.4\baselineskip}{.4\baselineskip}$\end{trivlist}}

 {\begin{trivlist}\item[]\textbf{Acknowledgments.}}{\end{trivlist}}

 {\begin{trivlist} \item[]{\bf\color{red} Comment. }}%
 {\color{red}\hspace*{\fill}$\rule{.4\baselineskip}{.4\baselineskip}$\end{trivlist}}

\makeatletter\@addtoreset{figure}{section}\makeatother

\makeatletter \@addtoreset{equation}{section} \makeatother

\newcommand{\rmO}{\mathrm{O}}
\newcommand{\rmo}{\mathrm{o}}

\newcommand{\rmd}{\mathrm{d}}
\newcommand{\rme}{\mathrm{e}}
\newcommand{\rmi}{\mathrm{i}}

\newcommand{\Rg}{\mathrm{Rg}\,}

\renewcommand{\leq}{\leqslant}
\renewcommand{\geq}{\geqslant}

\def\XXint#1#2#3{{\setbox0=\hbox{$#1{#2#3}{\int}$}
     \vcenter{\hbox{$#2#3$}}\kern-.5\wd0}}

\newcommand{ \Real} { \mathrm{Re}\,}
\newcommand{ \Imag} { \mathrm{Im}\,}
\newcommand{\bk}{ {\bf k} }
\newcommand{\eps}{\varepsilon}

\newcommand{\bbR}{\mathbb{R}}

\newcommand{\bbN}{\mathbb{N}}

\newcommand{\mmH}{\mathcal{H}}
\newcommand{\mmU}{\mathcal{U}}
\newcommand{\mmR}{\mathcal{R}}

\newcommand{\mmM}{\mathcal{M}}
\newcommand{\mmL}{\mathcal{L}}
\newcommand{\mmD}{\mathcal{D}}
\newcommand{\mmX}{\mathcal{X}}
\newcommand{\mmY}{\mathcal{Y}}

\title{Deformation of Striped Patterns by Inhomogeneities}

\author{Gabriela Jaramillo and Arnd Scheel\footnote{GJ and AS acknowledge support by the National Science
Foundation through grants DMS-0806614 and DMS-1311740.}
 \\
University of Minnesota \\
School of Mathematics \\
127 Vincent Hall, 206 Church St SE \\
Minneapolis, MN 55455, USA}

\begin{document}

\maketitle

\begin{abstract}
We study the effects of adding a local perturbation in a pattern forming system, taking as an example the Ginzburg-Landau equation with a small localized inhomogeneity in two dimensions. Measuring the response through the linearization at a periodic pattern, one finds an unbounded linear operator that is not Fredholm due to continuous spectrum in typical translation invariant or weighted spaces. We show that Kondratiev spaces, which encode algebraic localization that increases with each derivative, provide an effective means to circumvent this difficulty. We establish Fredholm properties in such spaces and use the result to construct deformed periodic patterns using the Implicit Function Theorem. We find a logarithmic phase correction which vanishes for a particular spatial shift only, which we interpret as a phase-selection mechanism through the inhomogeneity.
\end{abstract}

\section{Introduction}

\subsection{Pattern-forming systems}
Periodic, stripe-like patterns emerge in a self-organized fashion in a variety of experiments, ranging from Rayleigh-B\'enard convection to open chemical reactors. Such regular, periodic patterns are usually studied in domains with idealized periodic boundary conditions, where existence and stability can be readily obtained using classical methods of bifurcation theory. The simplest example for such pattern-forming systems is the Swift-Hohenberg equation
\[
u_t=-(\Delta +1)^2 u +\delta^2 u - u^3,\qquad (x,y) \in \bbR^2,
\]

which is known to possess periodic patterns of the form 
\[
u(x,y)=u_*(kx;k), u_*(\xi+2\pi;k)=u_*(\xi;k),
\]
with $k\sim 1$,
for  small $\delta$. 

Beyond periodic boundary conditions, the dynamics for $\delta\sim 0$ can be approximated using the amplitude equation formalism. In the case of the (isotropic) Swift-Hohenberg equation, one finds the Newell-Whitehead-Segel equation
\[
A_T=-(\partial_X-\rmi\partial_{YY})^2 A + A-A|A|^2,
\]
using an Ansatz
\[
u(x,y,t)=\delta A(\delta x,\delta y, \delta^2 t)\rme^{\rmi x}+c.c.,
\]

and expanding to order $\delta^3$, as a solvability condition \cite{cross1993pattern,mielke2002ginzburg}. There are several difficulties with the NWS equations and their validity as an approximation \cite{schneider1995validity}, related to the fact that the original equation is isotropic, while the expansion singles out a preferred wave vector, here the vector $k=(1,0)^T$. The situation is simplified in anisotropic pattern-forming systems such as 
\[
u_t = -(\Delta +1)^2u+\partial_{yy}u +\delta^2 u - u^3,
\]
where a similar Ansatz leads to the isotropic (sic!) Ginzburg-Landau equation
\begin{equation}\label{GL}
A_T=\Delta  A + A-A|A|^2,
\end{equation}
possibly after rescaling $X$ and $Y$. 

More drastically, one can approximate the dynamics near periodic patterns using an Ansatz
\[
u(t,x,y)=u_*(\Phi(\delta x,\delta y,\delta^2 t);|\nabla \Phi|),
\]
where one obtains as a compatibility condition the phase-diffusion equation 
\[
\Phi_T=\Delta_{X,Y}\Phi,
\]
for suitable values of the wavenumber $|\nabla \Phi|$, after possibly rescaling $X,Y$. 

Both, amplitude  and phase-diffusion equations can be shown to be good approximations under suitable choices of initial conditions, on time scales $T=\rmO(1)$; see for instance \cite{doelman2009dynamics,mielke2002ginzburg} and references therein.

\subsection{Inhomogeneities}

Local impurities in experiments sometimes have minor, sometimes more dramatic effects on the resulting patterns. It is known, for instance, that target-like patterns can nucleate at impurities in the Belousov-Zhabotinsky reaction; see \cite{kollar2007coherent}  for an analysis in this direction. Also, spiral wave anchoring at impurities such as arteries can have dramatic impact on excitable media \cite{munuzuri1998attraction}. Effects in Swift-Hohenberg-like systems appear to be more subtle and are the main focus of our present study. We focus on the somewhat simpler case of the isotropic GL equation \eqref{GL}, modeling anisotropic pattern-forming systems, with an added localized inhomogeneity,
\begin{equation}\label{e:cgli}
A_T=\Delta  A + A-A|A|^2+\varepsilon g(x,y).
\end{equation}

We intend to study inhomogeneities in the isotropic SH equation 
\[
u_t=-(\Delta +1)^2u+\delta^2 u - u^3+\varepsilon g(x,y),
\]
in future work. 

In order to illustrate the difficulties that arise, consider the more dramatic simplification of the phase-diffusion approximation with an inhomogeneity \footnote{We stress however that one cannot derive such an approximation in the presence of inhomogeneities due to the different scalings of $\Phi$ and $g$.},
\[
\Phi_t=\Delta \Phi+\varepsilon g(x,y).
\]
Stationary patterns here are solutions to the Poisson equation $\Delta\Phi=-\varepsilon g(x,y)$, with solutions exhibiting logarithmic growth at infinity. Thinking of the phase-diffusion as an approximation to a larger system, one would like a robust way of solving the stationary equation, if possible relying on an implicit function theorem. The difficulty here is that the Laplacian is not invertible on $L^2$, say, not even Fredholm as an unbounded operator. 

A remedy, in this context, are spaces with algebraic weights, which we refer to as Kondratiev spaces. To be precise, define $\langle \mathbf{x}\rangle=\sqrt{x^2+y^2+1}$ and  $M^{2,p}_{\gamma}$ as the completion of $C_0^{\infty}$ in the norm
\[
\|u\|_{M^{2,p}_{\gamma}}:=\sum_{j+k= 2}\|\langle {\bf  x} \rangle^{\gamma+2}\partial^j_x\partial^k_yu \|_{L^p}
+
\sum_{j+k= 1}\| \langle {\bf x} \rangle^{\gamma+1}\partial^j_x\partial^k_yu\|_{L^p}
+\| \langle {\bf x} \rangle^{\gamma} u\|_{L^p}.
\]
Note that the algebraic weights increase with the number of derivatives, making the different parts of the norms scale in the same fashion, as opposed to norms in the classical Sobolev spaces $W^{2,p}_{\gamma}$, with norm 
\[
\|u\|_{W^{2,p}_\gamma}:=\sum_{j+k\leq 2}\|\langle {\bf x} \rangle^{\gamma} \partial^j_x\partial^k_yu \|_{L^p}.
\]
It turns out that the Laplacian is Fredholm for suitable $\gamma$ on $M^{2,p}_{\gamma}$ \cite{McOwen}, albeit with negative index in dimension 2 when $\gamma>0$. This negative index makes ``explicit'' far-field corrections via logarithmic terms, just as seen in the Green's function of the Laplacian, necessary when describing the far-field effect of localized inhomogeneities on periodic patterns. This result also holds for a certain class of elliptic operators with coefficients that decay sufficiently fast at infinity \cite{lockhart1981fredholm}, but we are not aware of results for elliptic operators without scaling invariance. This represents a difficulty when looking at the linearization of \eqref{e:cgli}, which for $k=0$ decouples into $\Delta$ which is, and $\Delta -I$, which is not scale invariant. For $k \neq 0$ these components couple and simple scale invariance is lost. Furthermore, the linearization is in general not a small or compact perturbation of a scale invariant operator.

\subsection{Main results}

To state our main results, we consider the stationary solutions of \eqref{e:cgli},
\begin{equation}\label{stationaryGL}
0=\Delta A + A - A|A|^2+\varepsilon g(x,y).
\end{equation}
For $\varepsilon=0$, the system possesses ``stripe patterns'' 
\[
A(x,y)=\sqrt{1-k^2}\rme^{\rmi k x},
\]
for wavenumbers $|k|<1$. Those solutions are linearly stable for $|k|<1/\sqrt{3}$ and unstable for $|k|>1/\sqrt{3}$. The instability mechanism is known as the Eckhaus (sideband) instability. We are now ready to state our main result.

\begin{thm}\label{MainTh1}

Fix $k$ with $|k|<1/\sqrt{3}$ and suppose that $g\in W^{2,2}_{\beta}$ for some $\beta>2$. Then there exists a family of solutions to \eqref{stationaryGL},
\[
A(x,y;\varepsilon,\varphi)=S(x,y;\varepsilon,\varphi)\rme^{\rmi\Phi(x,y;\varepsilon,\varphi)},
\]
with $A(x,y;0,\varphi)=\sqrt{1-k^2}\rme^{\rmi (k x+\varphi)}$. Moreover, $A(x,y;\varepsilon,\varphi)$ is smooth in all variables and satisfies the following expansions in $x,y$ for fixed $\varepsilon,\varphi$,
\begin{align}
S(x,y;\eps,\varphi)&\to \sqrt{1-k^2}\\
\Phi(x,y;\eps,\varphi)-kx-\frac{c(\eps,\varphi)}{ 2k \sqrt{1-k^2}}  \log(\alpha x^2+y^2)&\to  \Phi_\infty(\eps)+\varphi,
\end{align}
for $|\mathbf{x}|\to\infty$, with $\alpha=\frac{1-k^2}{1-3k^2}$, for some smooth function $c(\eps,\varphi)$ with expansion
\[
c(\eps,\varphi)=\eps c_1(\varphi)+\rmO(\eps^2),
\]
where
\[
c_1(\varphi)=\frac{\sqrt{1-3k^2}}{\pi(1-k^2)} \iint \Imag[g(x,y)\rme^{-\rmi(kx+\varphi)}]  \rmd x \rmd y.
\]

\end{thm}

\begin{rem}

\begin{enumerate}
\item Our approach gives more detailed expansions than stated. In fact, we obtain a decomposition of $S$ and $\Phi$ into a localized part that is smooth in $\varepsilon$, uniformly in $\mathbf{x}$, and an explicit logarithmic far-field correction with coefficient $c(\eps,\varphi)$; see the Ansatz \eqref{Ansatz}. Also, in the class of functions with this particular form, the solutions described in the theorem are locally unique. 

\item The expression for $c_1(\eps,\varphi)$ is reminiscent of a projection onto the kernel. Indeed, the integral represents the scalar product $(u,v)=\Real \int u\bar{v}$ of the perturbation $\eps g$ and the ``kernel'' of the linearization induced by the infinitesimal  phase rotation, $\frac{\rmd}{\rmd\varphi}\rme^{\rmi {kx+\varphi}}$. Vanishing of such a scalar product indicates persistence of solutions in problems where the linearization is Fredholm, such as in the Melnikov analysis for homoclinic bump-like solutions. 

In fact,  $c_1$ possesses at least 2 zeroes. Assuming that $c_1(\varphi_*)=0$, $c_1'(\varphi_*)\neq 0$, we can find $\varphi_*(\eps)$ so that
\[
c(\eps,\varphi_*(\eps))=0.
\]
Inspection of the expansion in the theorem shows that $\varphi$ is the phase shift of the underlying pattern in the far field. For these specific values of $\varphi$, the correction in the far field to the periodic pattern is bounded and small, while for other values of $\varphi$ the correction is unbounded in the phase. We interpret this result by referring to $\varphi_*(\eps)$ as the \emph{selected phase}. In other words, introducing inhomogeneities induces a selected phase shift of the primary pattern which accommodates stationary solutions without logarithmic corrections. Numerical simulations confirm this phenomenon, with a diffusive spread of the phase shift in the domain. It would be interesting to establish this diffusive convergence to a selected phase analytically.

\item Very similar results hold in space dimensions 1 and 3. In one space dimension, the analysis is easier since ODE methods can be used to analyze stationary solutions. In fact, the analysis reduces to a Melnikov analysis for the intersection of center-stable and center-unstable manifolds. In this one-dimensional context, the analysis also immediately carries over to the Swift-Hohenberg equation. On the other hand, the 3-dimensional case is easier than the two-dimensional case considered here since the corrections needed to compensate for negative Fredholm indices are decaying like $1/|\mathbf{x}|$. In fact, the Laplacian is invertible for suitable weights $\gamma$ in Kondratiev spaces in $\bbR^3$.

\end{enumerate}

\end{rem}

The structure of this paper is as follows. In Section \ref{weightedspaces} we give a more detailed description of weighted Sobolev spaces and Kondratiev spaces, and state Fredholm properties of the Laplacian in this setting. Next, in Section \ref{summary}, we summarize the procedure leading up to the main result, first explaining the difficulties encountered when analyzing the linearization of equation \eqref{stationaryGL} about stripe patterns, and then describing the linear operator $T$ that we use to overcome these difficulties. In Section \ref{sectionlinear} we use the results from Section \ref{weightedspaces} to explore the Fredholm properties of this last operator and show that by adding logarithmic corrections we can obtain an invertible operator $\hat{T}$. Section \ref{sectionnonlinear} establishes properties of the nonlinearity in our functional analytic setting. We show that the full operator $\hat{F}$  is  well defined, continuously differentiable, with invertible linearization
$\hat{T}$. Finally, in Section \ref{Mainresult}, we prove our main result using the Implicit Function Theorem.

\section{Fredholm properties of the linearization and weighted spaces} \label{weightedspaces}

This section is intended as a summary of theorems and results that describe weighted spaces and their properties. These results are the basis for the analysis in the following sections and will allow us to conclude Fredholm properties of the linearized operators considered therein.

\subsection{Weighted Sobolev spaces}
Throughout the paper we will use the symbol $W^{s,p}_{\gamma}$ to denote  weighted Sobolev spaces, which we define as the completion of $C^{\infty}_0(\bbR^n)$ under the norm
\[ \| u\|_{W^{s,p}_{\gamma}} = \sum_{|\alpha|\leq s} \| \langle {\bf x} \rangle^{\gamma} D^{\alpha} u \|_{L^p}.\]

Here, $\langle {\bf x} \rangle = ( 1 + |{\bf x}|^2)^{1/2}$, ${\bf x}=(x,y)$, $\gamma \in \bbR$, and $s$ is a positive integer. 

We start with a generalization of the invertibility of the operator $\Delta- I$ to weighted spaces.

\begin{prop}\label{Sobolev}
The operator $\Delta -a : W_{\gamma}^{2,p} \rightarrow L^p_{\gamma}$ is invertible for all real numbers $a >0$ and $p \in [1, \infty)$. 
\end{prop}

\begin{lem}\label{notFredholm}
 The operator $\Delta_{\gamma}:W^{2,p}_{\gamma} \rightarrow L^p_{\gamma}$ is not a Fredholm operator for $p \in [1, \infty)$.
\end{lem}

Both results follow in a straightforward fashion by noticing that $L^2_\gamma$ and $L^2$ are conjugate by a multiplication operator, and the conjugate operator to $\Delta$ is a compact perturbation of the Laplacian, which establishes Fredholm properties.

\subsection{ Kondratiev spaces}

Kondratiev spaces were first introduced to study boundary value problems for elliptic equations in domains with conical points \cite{kondrat1967boundary}.  Nirenberg and Walker \cite{nirenberg1973null} later used these spaces to show that  elliptic operators with coefficients that decay sufficiently fast at infinity have finite dimensional kernel when considered as operators between these weighted spaces and McOwen \cite{McOwen} established  Fredholm properties for the Laplacian. Lockhart and McOwen \cite{lockhart1981fredholm, lockhart1983elliptic, lockhart1985elliptic} built on these ideas to establish Fredholm properties for classes of elliptic operators. For example, Lockart \cite{lockhart1985elliptic} studied elliptic operators of the form $A = A_{\infty} +A_0$ in non-compact manifolds, where $A_{\infty}$  represents a constant coefficient homogeneous elliptic operator of order $m$, and $A_0$ an operator of order at most $m$ with coefficients that decay fast at infinity. More recently, Kondratiev spaces 
were used to study the Laplace operator in exterior domains \cite{amrouche2008mixed} and similar weighted spaces were used in \cite{milisic2013weighted} to understand the Poisson equation in a 1 periodic infinite strip $Z = [0,1] \times \bbR$. 

We will denote Kondratiev spaces by $M^{s,p}_{\gamma} $ and define them as the completion of $C^{\infty}_0(\bbR^n)$ under the norm
\[ 
\| u\|_{M^{s,p}_{\gamma}} = \sum_{|\alpha|\leq s} \| \langle {\bf x} \rangle^{\gamma + |\alpha|} D^{\alpha} u \|_{L^p},
\]
where $\langle {\bf x} \rangle = ( 1 + |{\bf x}|^2)^{1/2}$, $\gamma \in \bbR$, $ s \in \bbN$, and $p \in (1, \infty)$. Notice the embeddings $M^{s,p}_{\gamma} \hookrightarrow W^{s,p}_{\gamma}$, as well as $M^{s,p}_{\gamma} \hookrightarrow M^{s-1,p}_{\gamma}$. 

The following theorem describes the behavior of the Laplacian in Kondratiev spaces. Its proof can be found in \cite{McOwen}.
 
\begin{thm}\label{McOwen}
Let $1<p=\frac{q}{q-1}<\infty$,  $n \geq 2$, and $\gamma \neq -2 + n/q + m $ or $\gamma \neq -n/p -m $, for some $m \in \bbN$.  Then 
\[ 
\Delta: M^{2,p}_{\gamma} \rightarrow L^p_{\gamma+2}, 
\]
is a Fredholm operator and
\begin{enumerate}
\item for $-n/p < \gamma < -2 + n/q$ the map is an isomorphism;
\item for $-2 + n/q + m < \gamma< -2 + n/q + m+1$ , $m \in \bbN$, the map is injective with closed range equal to 
\[
R_m = \left\{ f \in L^p_{\gamma +2} : \int f(y)H(y) =0 \  \text{for all } \ H \in \bigcup_{j=0}^m \mmH_j\right\}; 
 \]
\item for $-n/p - m -1 < \gamma <-n/p -m $, $m \in \bbN$, the map is surjective with kernel equal to 
\[ N_m = \bigcup_{j=0}^m \mmH_j
.\]
\end{enumerate}
Here,  $\mmH_j $ denote the harmonic homogeneous polynomials of degree $j$.

On the other hand, if $\gamma = -n/p - m $ or $\gamma = -2 + n/q + m$ for some $m \in \bbN$, then $\Delta$ does not have closed range.
 \end{thm}

\section{Outline of proof}\label{summary}

Recall that we are interested in solutions of \begin{equation}\label{gl}
0 = \Delta A + A - A|A|^2 + \eps g(x,y),
\end{equation}
for localized $g$ and $\varepsilon$ small, close to the solutions at $\eps=0$,  $A_*(x)= \sqrt{1-k^2}\rme^{\rmi kx}$,  $|k|^2 < 1/3$. Since the case $k=0$ is in fact easier, we will assume in the following that $k>0$. A reasonable Ansatz then is
\[ 
A(x,y, \eps) =  (\sqrt{1-k^2} + s(x,y;\eps) ) \rme^{\rmi(kx + \phi( x,y;\eps))},
\]
with new variables $s,\phi$, which solve
\begin{align}\label{amp1}
\Delta s + (s+\tau) - ( s + \tau)( k^2 + 2k \partial_x \phi + | \nabla \phi|^2)  - (s + \tau)^3 + \eps \Real( g \rme^{-\rmi(kx + \phi)})&=0\\
\Delta \phi + \frac{2k \partial_x s}{s + \tau} + \frac{2 \nabla s \cdot \nabla \phi}{s +\tau} + \frac{ \eps \Imag(g \rme^{-\rmi(kx + \phi)})}{s + \tau}&=0,\label{phase1}
\end{align}
where we set $\tau=\sqrt{1-k^2}$. Linearizing at $\eps=0, s=0,\phi=0$, we obtain the operator
\begin{equation}\label{linearoperator}
 L \begin{bmatrix} s\\ \phi \end{bmatrix} = \begin{bmatrix} 
\Delta  - 2\tau^2 & -2k\tau   \partial_x  \\ 
 \dfrac{2k }{\tau}  \partial_x  & \Delta  
\end{bmatrix} \begin{bmatrix} s\\ \phi \end{bmatrix}.
\end{equation}

The results from the Section \ref{weightedspaces}, and in particular Theorem \ref{McOwen}, suggest that we should require that $\phi\in M^{2,p}_{\gamma}$ and that \eqref{phase1} holds in $L^p_{\gamma+2}$. 
Then $\phi_x\in W^{1,p}_{\gamma+1}$ and, using the linearization of \eqref{amp1} with Proposition \ref{Sobolev}, this suggests $s\in W^{3,p}_{\gamma+1}$ and $s_x\in W^{2,p}_{\gamma+1}$. This however is not sufficient localization for \eqref{phase1}, where $s_x$ enters, and which we assumed to be satisfied in $L^p_{\gamma+2}$. 

In other words, the coupling terms, which are absent for $k=0$, prohibit the simple use of Sobolev spaces for $s$ and Kondratiev spaces for $\phi$. Roughly speaking, the coupling destroys the linear scaling invariance in the $\phi$-equation, which is necessary at least at infinity in results on Fredholm properties in Kondratiev spaces, which intrinsically mix regularity and localization properties. We intend to address these problems more generally in future but focus here on a simple an direct construction that circumvents the problem by extending the system and introducing appropriate norms for derivatives. 

Consider therefore  $ T: \mmD \subset \mmX  \rightarrow \mmR \subset \mmY$,
\begin{equation}\label{extendedlinear}
  T \begin{bmatrix} s \\ \psi \\ \theta \\ u \\ v\\ w \end{bmatrix} = 
 \begin{bmatrix}
 \Delta -a & -1 & 0 & 0 & 0 & 0 \\
 0 & \Delta & 0 & b & 0 & 0\\
 0 & 0 & \Delta & 0 & b & 0\\
 0 & -\partial_{xx} & 0 & \Delta -a & 0 & 0 \\
 0 & 0 & -\partial_{xx} & 0 & \Delta -a & 0 \\
 0 & -\partial_{yy} & 0 & 0 & 0 & \Delta -a 
 \end{bmatrix}\begin{bmatrix} s \\ \psi \\ \theta \\ u \\ v\\ w \end{bmatrix},
\end{equation}
were \hskip0.3cm $\psi = 2k \tau\partial_x \phi$,\hskip0.3cm $\theta= 2k \tau \partial_y \phi$,\hskip0.3cm $a = 2\tau^2$,\hskip0.3cm $b = 4k^2$, 
\[
\mmX = W^{2,2}_{\gamma} \times M^{2,p}_{\gamma} \times M^{2,p}_{\gamma} \times L^p_{\gamma+2}\times L^p_{\gamma+2}\times L^p_{\gamma+2},\] 
\[ \mmY = L^p_{\gamma} \times L^p_{\gamma+2} \times L^p_{\gamma+2} \times W^{-2,p}_{\gamma+2}\times W^{-2,p}_{\gamma+2}\times W^{-2,p}_{\gamma+2}.
\]
Here, $W^{-k,p}_\gamma$ denotes the dual of $W^{k,p}_\gamma$. We also define the closed subspaces
 \[ \mmD = \left \{ (s, \psi, \phi, u,v,w) \in \mmX: u = \partial_{xx} s, \quad v= \partial_{xy} s,\quad w = \partial_{yy} s, \quad \partial_y \psi = \partial_x \theta \right \}, \]
\begin{align*}
 \mmR & = \left \{ y=(f_1,f_2,f_3,f_4,f_5,f_6) \in \mmY: \int f_2 = \int f_2 \cdot y= \int f_3=  \int f_3 \cdot x=0, \right. \\
& \quad \left.f_4 = \partial_{xx} f_1,  f_5= \partial_{xy} f_1,  f_6 = \partial_{yy}f_1, \quad \text{and} \quad \partial_y f_2 = \partial_x f_3   \right \}.
\end{align*}
The second and third equation in (\ref{extendedlinear}) are obtained by taking the $x$ and $y$ derivatives of the phase equation. The last three equations come from taking the second order partial derivatives of the amplitude equation with respect to $xx$, $xy$, and $yy$.
 
 We will see in Section \ref{sectionlinear} that the linear operator $T: \mmD \rightarrow \mmR$ is a Fredholm operator of index $-1$ for optimal choices of weights, indicating a missing parameter in the far field. We therefore add a single degree of freedom in the far field through the variable $\hat{c}\in\bbR$ via the Ansatz  
\[
\begin{array}{l c l c c c l c l}
s & =& \hat{s} + \hat{c} P_1,  & & &  &u &=& \hat{u} + \hat{c} \partial_{xx} P_1,\\
\psi & =& \hat{\psi} + \hat{c} \partial_x P_2,  & & &  &v &=& \hat{v} + \hat{c} \partial_{xy} P_1,\\
\theta & =& \hat{\theta} + \hat{c}\partial_y P_2,  & & &  &w &=& \hat{w} + \hat{c} \partial_{yy} P_1,
\end{array}
\]
where 
\begin{equation}\label{e:p1p2}
P_1= \dfrac{1- \alpha}{2 b \alpha}\partial_x[\chi \log(\alpha x^2+y^2)], \qquad P_2 = \dfrac{1}{2}\chi \log(\alpha x^2+y^2),
\end{equation}
$b = (2k)^2$, $\alpha = \dfrac{1-k^2}{1-3k^2}$, and $\chi$ is a smoothed version of the indicator function $\chi_{|\mathbf{x}|>1}$. Substituting this Ansatz into \eqref{amp1},\eqref{phase1} and linearizing, we find an operator $\hat{T}: \mmD \times \bbR \rightarrow \mmR$, given by 
\begin{equation}\label{newT} \hat{T} \xi =
  \begin{bmatrix} 
    \Delta -a & -1 &  & & &  &\Delta P_1  \\
 & \Delta & &b & & &\Delta   P_2 +b \partial_{xx}P_1  \\
 & &\Delta &  & b& &  \Delta  P_3 + b \partial_{xy} P_1 \\
 & - \partial_{xx} &  & \Delta -a & &  & \Delta  \partial_{xx} P_1\\
 &  &- \partial_{xx}  & & \Delta -a  &  & \Delta \partial_{xy} P_1    \\
& - \partial_{yy} &  &  &  &\Delta -a &  \Delta  \partial_{yy} P_1 \\
 \end{bmatrix} \begin{bmatrix} \hat{s} \\ \hat{\psi} \\ \hat{\theta} \\ \hat{u} \\ \hat{v} \\ \hat{w} \\ \hat{c}    \end{bmatrix},
 \end{equation}
where  again $a = 2\tau^2$, and $b = 4k^2$. We will show in the Section \ref{sectionlinear} that this operator is invertible. 

Recapitulating, we are lead to consider the Ansatz
\begin{equation}\label{Ansatz}
A(x,y; \eps, \varphi) = \left(\sqrt{1 -k^2} + s(x,y;\eps, \varphi)+ c(\eps, \varphi) P_1(x,y)\right) \rme^{ \rmi \left (kx +  \phi(x,y;\eps, \varphi)+  \frac{c(\eps, \varphi)}{2k \sqrt{1-k^2}} P_2(x,y) \right)}, 
\end{equation}
with $P_1$ and $P_2$ as in \eqref{e:p1p2}, with additional equations for the derivatives 
\[ 
u = \partial_{xx} s, \quad v = \partial_{xy} s, \quad w = \partial_{yy} s, \quad \psi = (2k)  \partial_{x} \phi,\quad \theta = (2k)  \partial_y \phi.
\]
We obtain a nonlinear equation  
\begin{equation}\label{nl0}
\hat{F}_{\eps, \varphi}=0, \qquad \hat{F}_{\eps, \varphi} : \mmD \times \bbR^3 \rightarrow \mmR;
\end{equation}
see Subsection \ref{sectionnonlinear} for a more detailed description of this nonlinear equation. The advantage of this subtle reformulation of the problem is encoded in the following result, which establishes applicability of the standard Implicit Function Theorem and is the key ingredient to the proof of Theorem \ref{MainTh1}.

\begin{thm}\label{MainTh}
Let $p=2$, $\gamma \in (0,1)$, and $g \in W^{2,p}_{\beta}$, with $\beta > \gamma+2$. Then, the operator $\hat{F}_{\eps,\varphi}: \mmD \times \bbR^3 \rightarrow \mmR$  is of class $C^\infty$. Furthermore, for fixed $\varphi$ and at $\eps =0$, its derivative is given by the invertible operator $\hat{T}: \mmD \times \bbR \rightarrow \mmR$.
\end{thm}

The proof of this theorem will occupy the next 2 Sections. In Section \ref{sectionlinear} we show the fact that $T$ is a Fredholm operator and that $\hat{T}$ is invertible, and in Section \ref{sectionnonlinear} we show that the operator $\hat{F}$ is of class $C^{\infty}$.

\section{The linear operator}\label{sectionlinear}
In this section we consider the linear operators $T: \mmD \rightarrow \mmR$ and $\hat{T}:\mmD \times \bbR \rightarrow \mmR$ defined in (\ref{extendedlinear}) and (\ref{newT}), respectively. We first prove that for $p=2$ and $\gamma \in (0,1)$ the operator $T: \mmX \rightarrow \mmY $ is a Fredholm operator of index $-6$.  Then, we show  that restricting the domain and range to $\mmD$ and $\mmR$ turns $T$ into a Fredholm operator of index $-1$. Finally, using a bordering lemma, we prove  that the operator $\hat{T}$ is invertible. Throughout, we assume $ \gamma \in (0, 1)$ and $0<|k|< \dfrac{1}{\sqrt{3}}$.

\begin{prop}\label{FredholmT1} \
The operator $T:\mmX \rightarrow \mmY$ defined in (\ref{extendedlinear}) is a Fredholm operator with index $i=-6$ and trivial kernel. The cokernel is spanned by 
\begin{equation}\label{sol}
\{(0,a,0,b,0,0)^Te_j^*, (0,0,a,0,b,0)^Te_j^*, \, j=1,2,3\},\qquad e_1^*=1,e^*_2=x,e^*_3=y.
\end{equation}

\end{prop}

\begin{proof} 
First, notice that due to the lower block-triangular structure of $T$, it is enough to consider the restriction $\tilde{T}$ to the variables $\psi, \theta, u$, $v$, which we write in the form 
\begin{equation}\label{Perturbed} 
\tilde{T}\xi=  L \xi+b  B\xi, 
\end{equation}
where $b = 4k^2$, $a = 2(1-k^2)$, 
\[ 
L=  \begin{bmatrix}  \Delta  & 0 &0&0 \\ 0& \Delta&0 &0\\  - \partial_{xx} &0& \Delta -a&0 \\  0 &-\partial_{xx} &0 &\Delta -a \end{bmatrix}, \qquad
B = \begin{bmatrix} 0 &0 & I& 0 \\ 0 & 0 &0 &I\\ 0&0&0&0\\0&0&0&0 \end{bmatrix}, \qquad
\text{and} \quad\xi = \begin{bmatrix} \psi \\ \theta \\ u \\ v \end{bmatrix}. 
\]
We need to show that  
\[
\tilde{T}:   \tilde{\mmX} = M^{2,2}_{\gamma} \times M^{2,2}_{\gamma} \times L^2_{\gamma+2} \times L^2_{\gamma+2}\longrightarrow\tilde{\mmY} = L^2_{\gamma+2} \times L^2_{\gamma+2} \times W^{-2,2}_{\gamma+2} \times W^{-2,2}_{\gamma+2},
\]
is a Fredholm operator of index $i =-6$. Since $\tilde{T}$ is block-diagonal with respect to $(\psi,u)$ and $(\theta,v)$, it is sufficient to show that the restriction to $(\psi,u)$ is Fredholm with index $-3$. For $b=0$, the claim now follows directly from Theorem \ref{McOwen} and Proposition \ref{Sobolev}, due to the lower triangular structure and the fact that, with our choice of $\gamma,p$, the Laplacian is Fredholm with index $-3$. We will address the more difficult situation $b\neq 0$, next.

In order to establish the desired Fredholm properties, we need to solve 
\begin{align}\label{e1}
 \Delta \psi + b u & = f_1\\ \label{e2}
  \partial_{xx} \psi + (\Delta -a) u  &= f_2,
\end{align}
for $f_1,f_2$ in a codimension-3 subspace of $L^2_{\gamma+2} \times W^{-2,2}_{\gamma+2}$, with bounds on $(\psi,u)\in M^{2,2}_{\gamma} \times L^2_{\gamma+2}$.

Denote by $I-Q$ a projection on the range of the Laplacian, so that $\int (I-Q) f=\int x  (I-Q) f=\int y  (I-Q) f =0$. We can then decompose 
\begin{align}\label{e3}
 \Delta \psi + b u & = (I-Q)f_1\\ \label{e4}
  \partial_{xx} \psi + (\Delta -a) u  &= (I-Q)f_2,
\end{align}
and $Q u=\frac{1}{b} Q f_1= - \frac{1}{a}Q f_2$, exhibiting the 3 solvability conditions \eqref{sol}. We next solve \eqref{e4} for $u$, substitute in \eqref{e3}, to obtain
\[
\mmL \psi=(I-Q)f_1-(\Delta-a)^{-1}(I-Q)f_2=:f,\qquad \mmL=[\Delta + b (\Delta -a)^{-1} \partial_{xx}],
\]
where $f=(I-Q)f$. It therefore remains to show that $\mmL:M^{2,2}_{\gamma}\to (I-Q)L^2_{\gamma+2}$ is invertible.
We therefore factor 
\[ \mmL \psi = \mmM \left( \Delta - \frac{b}{a} \partial_{xx}\right) \psi, \]
\begin{center}
\begin{tikzpicture} 
\matrix(m)[matrix of math nodes,
 row sep=3.5em, column sep=5.5em,
  text height=2.5ex, text depth=0.25ex]
   {M^{2,2}_{\gamma}&(I-Q) L^2_{\gamma+ 2}&(I-Q) L^2_{\gamma+2} \\};
    \path[->,font=\footnotesize,>=angle 90] 
    (m-1-1) edge node[auto] {$\Delta - \dfrac{b}{a} \partial_{xx}$} (m-1-2)
	 (m-1-2) edge node[auto] {$\mmM$} (m-1-3); 
\end{tikzpicture}
\end{center}
By Theorem \ref{McOwen} the operator $\left(\Delta - \dfrac{b}{a} \partial_{xx} \right) : M^{2,2}_{\gamma} \rightarrow R_m$ is invertible, since it is conjugate to the Laplacian by a simple $x$-rescaling operator. It is therefore sufficient to establish that $\mmM$ is an isomorphism of $(I-Q)L^2_{\gamma+2}$. 
Consider therefore the associated Fourier symbol 
\[
\hat{\mmM}(k,l) = \frac{k^2+l^2}{k^2+l^2 -\frac{b}{a} k^2} - \frac{bk^2}{(k^2+l^2+a) ( k^2+l^2 - \frac{b}{a}k^2)}.
\]
Exploiting that $k^2<\frac{1}{3}$ so that  $1-\dfrac{b}{a} >0$, it is straightforward to see that 
\[
\sup_{(k,l) \in \bbR^2} |\hat{\mmM}(k,l)|+|\hat{\mmM}(k,l)^{-1}|<\infty,
\]
so that $\mmM$ is an isomorphism of $L^2$. We next show that $\mmM$ is an isomorphism on $(I-Q)L^2_j$, $j=2,3$, which by interpolation theory gives the desired result. Equivalently, we need to show boundedness of the multiplication operator $\hat{\mmM}$ on the subspace of  $H^j$, $j=2,3$ consisting of functions functions $f$ with $f(0)=0$ and, in case $j=3$, $\nabla f(0)=0$. Since $\hat{\mmM}(k,l)=a+\rmO(k^2+l^2)$ near $k=l=0$, we readily find that  $\| D^{\alpha} \hat{\mmM } |_{L^{\infty} }+\| D^{\alpha}( \hat{\mmM }^{-1}) |_{L^{\infty} }< \infty$ for all indices $|\alpha|\leq 2$, which proves that $\hat\mmM$ is an isomorphism on $H^2$. For the case $j=3$, we use that $\| D^{\alpha} (\bk\hat{\mmM }) |_{L^{\infty} }+\| D^{\alpha}(\bk \hat{\mmM }^{-1}) |_{L^{\infty} }< \infty$ for all indices $|\alpha|=3$, which readily implies that $\hat{\mmM}$ is an isomorphism on $H^3\cap\{f(0)=0\}$. One can also readily check that the range of $\mmM$ and $\mmM^{-1}$ is indeed contained in $\Rg(I-Q)$, which concludes the proof.
\end{proof}

\begin{cor}\label{FredholmT2}
The operator  $T: \mmD \rightarrow \mmR$ defined by (\ref{extendedlinear}) is Fredholm with index $-1$ and cokernel $\partial_y f_2+\partial_x f_3$.
\end{cor}

\begin{proof} 
Inspection of $T$ shows that the range of the restriction of $T$ to $\mmD$ is actually contained in $\mmR$, which implies that $T:\mmD\to\mmR$ is injective and the range is closed ($T$ is semi-Fredholm). We need to show that the cokernel is one-dimensional. 

Take $f\in \Rg(T)\subset \mmX$, with $f_4 = \partial_{xx} f_1,\,  f_5= \partial_{xy} f_1, \, f_6 = \partial_{yy}f_1$, and $\partial_y f_2 = \partial_x f_3$. 
By construction of $T$, notably having taken derivatives of equations for $s$ and $\phi$, and by injectivity, $T^{-1}f$ satisfies
$u = \partial_{xx} s, \,v= \partial_{xy} s$, $w = \partial_{yy} s$, and $\partial_y \psi= \partial_x \theta$. As a consequence, the cokernel of $T:\mmD\to\mmR$ is a subset of the cokernel of $T:\mmX\to\mmY$. Inspecting the cokernel in \eqref{sol} and the definition of $\mmR$, we see that $(e_j^*,f_4)=(e_j^*,f_5)=0$, so that the integral conditions in the definition of $\mmR$ represent precisely 4 of the 6 conditions on the co-kernel. One of the remaining conditions, $\int f_2\cdot x=\int f_3\cdot y$ is a consequence of $\partial_y f_2=\partial_x f_3$, whereas $\int f_2\cdot x$ can be readily seen to act non-trivially. As a consequence $T$ is Fredholm of index -1 as claimed, and the cokernel is spanned by $(0,x,0,0,0,0)^T$ or, equivalently,  $(0,x,y,0,0,0)^T$.
\end{proof}

We  next consider the operator $\hat{T}: \mmD \times \bbR \rightarrow \mmR$ defined by (\ref{newT}). Recall that
\begin{equation}\label{changevariables}
\begin{array}{l c l c c c l c l}
s & =& \hat{s} + \hat{c} P_1,  & & &  &u &=& \hat{u} + \hat{c} \partial_{xx} P_1,\\
\psi & =& \hat{\psi} + \hat{c} P_2,  & & &  &v &=& \hat{v} + \hat{c} \partial_{xy} P_1,\\
\theta & =& \hat{\theta} + \hat{c} P_3,  & & &  &w &=& \hat{w} + \hat{c} \partial_{yy} P_1.
\end{array}
\end{equation}

Here, 
\[P_1(x,y) =  \dfrac{ (1-\alpha)}{2b\alpha} \partial_x [\chi \ln( \alpha x^2+y^2)],\quad
  P_2(x,y)  =  \dfrac{1}{2} \partial_x [ \chi \ln( \alpha x^2+y^2)],\quad
   P_3(x,y) =  \dfrac{1}{2} \partial_y [ \chi \ln( \alpha x^2+y^2) ],
\]
%

 with $\alpha  = \dfrac{1-k^2}{1-3k^2}$, \hskip0.4cm $b= (2k)^2$, \hskip0.4cm and  \hskip0.4cm $\chi \in C^{\infty}(\bbR^2)$ defined by
 \[\chi(x,y) = \left\{ \begin{array}{c c l} 0 & \text{if} & 0 \leq \sqrt{\alpha x^2+y^2} \leq 1/2\\
 						1 & \text{if} & 1 \leq \sqrt{\alpha x^2+y^2}
						\end{array} \right. .\]

To show $\hat{T}: \mmD \times \bbR \rightarrow \mmR$ is invertible we will need the following lemma.

\begin{lem}\label{extraM}

The operator 
\[
M: \bbR \rightarrow \mmR,\qquad  c \mapsto \begin{bmatrix}  \Delta  P_1&
							\Delta    P_2 +b \partial_{xx} P_1  &
							 \Delta  P_3 + b \partial_{xy} P_1 &
							 \Delta \partial_{xx} P_1&
							 \Delta \partial_{xy} P_1&
							 \Delta \partial_{yy} P_1
							 							 \end{bmatrix}^T c,
\]
is well-defined and its range satisfies 
\[ \iint \Big [\Delta P_2 +b \partial_{xx} P_1 \Big ] \cdot x \rmd x\rmd y= \iint \Big [ \Delta P_3 +b \partial_{xy}P_1 \Big ] \cdot y\rmd x\rmd y \neq 0 . \]
\end{lem}

\begin{proof} First notice that the smooth functions $P_i$, for $i = 1,2,3$, are bounded in compact sets and behave like $\dfrac{1}{|{\bf x}|}$ for large values of $|{\bf x} |$ so that the range of $M$ is indeed a subset of $\mmY$. From the definition, it is not difficult to check that the operator $M$ maps into the desired space $\mmR$.  We need to show that
\[ \iint \Big [\Delta P_2 +b \partial_{xx} P_1 \Big ] \cdot x\rmd x\rmd y = \iint \Big [ \Delta P_3 +b \partial_{xy}P_1 \Big ] \cdot y\rmd x\rmd y \neq 0 . \]

Straightforward calculations, using the rescaling $X = \sqrt{\alpha} x , Y = y$, show that 
\[ \Delta  P_2+ b \partial_{xx}  P_1 = \frac{\sqrt{\alpha}}{2} \Delta_{X,Y} \left [ \frac{\partial}{\partial X} \left ( \chi \cdot \ln ( X^2+Y^2)\right)  \right ],\]
where we write $\Delta_{X,Y} = \partial_{XX} + \partial_{YY}$. Therefore,
\begin{align*}
 \iint [\Delta  P_2+ b \partial_{xx}  P_1 ] \cdot x  \rmd x \rmd y &= \iint \left[  \frac{\sqrt{\alpha}}{2} \Delta_{X,Y} \left [ \frac{\partial}{\partial X} \left ( \chi \cdot \ln ( X^2+Y^2)\right)  \right ] \right] \cdot X  dX dY  \\
 &= \iint \left[  \frac{\sqrt{\alpha}}{2} \Delta_{X,Y}  \left ( \chi \cdot \ln ( X^2+Y^2)\right)  \right]   dX dY  \\
 &= \sqrt{\alpha} \pi .
 \end{align*}

Similarly, using the same rescaling, it can be shown that
\[
 \Delta  P_3+ b \partial_{xy} P_1 = \frac{1}{2} \Delta_{X,Y} \left[ \frac{\partial}{\partial Y} \left( \chi \cdot \ln (X^2+Y^2)  \right)  \right],\]
and consequently
\begin{align*}
\iint \left[  \Delta P_3+ b \partial_{xy}  P_1 \right] \cdot y \rmd x\rmd y &= \iint \left[ \frac{1}{2} \Delta_{X,Y} \left[ \frac{\partial}{\partial Y} \left( \chi \cdot \ln (X^2+Y^2)  \right)  \right] \right] \cdot Y \sqrt{\alpha} dXdY \\
&= \sqrt{\alpha} \pi.
\end{align*}
\end{proof}

\begin{cor}\label{FredholmT3}
The operator $\hat{T}:\mmD \times \bbR \rightarrow \mmR$ is invertible.
\end{cor}

\begin{proof} Notice that $\hat{T}= [ T  M]$, where $T: \mmD \rightarrow \mmR$ is the Fredholm operator of index $-1$ described in Corollary \ref{FredholmT2}, and  $M: \bbR \rightarrow \mmR$ is defined in Lemma \ref{extraM}. A bordering lemma implies that $\hat{T}: \mmD \times \bbR \rightarrow \mmR$ is a Fredholm operator of index $0$. Lemma \ref{extraM} implies that $\Rg(M)\not\subset\Rg(T)$, so that $\hat{T}$ is onto, hence invertible.
\end{proof}

\section{The nonlinear map}\label{sectionnonlinear}

In this section, we show by a series of lemmas that the nonlinear problem \eqref{nl0} is well defined and continuously differentiable. We first give explicit expressions for each component of  the nonlinearities. We then state and prove several lemmas that will help us show that the nonlinearity is well defined. Finally, we show that the nonlinearities are of class $C^{\infty}$.

 The following expressions represent each component of the operator $\hat{F}_{\eps, \varphi}$ announced in \eqref{nl0}:
\begin{align*}
 \hat{F}_1 (\xi;\eps, \varphi) =& \Delta s - 2\tau^2 s - (s + \tau) \left( \frac{\psi}{\tau} + \frac{1}{(2k\tau)^2} (\psi^2 + \theta^2)\right) - (s^3 + 3s^2 \tau)+ \eps \Real[ g \rme^{-\rmi (kx +\phi(\varphi))}],\\
\hat{F}_2 (\xi;\eps, \varphi) =& \Delta \psi + \frac{(2k)^2\tau u + 2( u \psi + v \theta + s_x \psi_x + s_y \psi_y) }{s+ \tau} -\frac{(2 k)^2 \tau(s_x)^2 + 2s_x( s_x \psi +s_y \theta)}{ (s+\tau)^2} \\
 &+ \frac{\eps \Imag[ \partial_x( g \rme^{-\rmi (kx +\phi(\varphi))})]}{s+ \tau} -\frac{\eps s_x \Imag[g \rme^{-\rmi(kx +\phi(\varphi))}] }{(s+ \tau)^2},
\\
\hat{F}_3 (\xi;\eps, \varphi) =& \Delta \theta + \frac{(2k)^2 \tau v + 2( v \psi + w \theta + s_x \theta_x +s_y \theta_y) }{s + \tau} - \frac{(sk)^2 \tau s_xs_y + 2s_y(s_x \psi + s_y \theta)}{(s+\tau)^2} \\
 &+ \frac{ \eps \Imag[ \partial_y( g \rme^{-\rmi(kx +\phi(\varphi))})]}{s+\tau} - \frac{ \eps s_y \Imag[ g \rme^{-\rmi(kx + \phi(\varphi))}]}{(s+ \tau)^2},
\\
\hat{F}_4 (\xi;\eps, \varphi) =& \Delta u - 2 \tau^2 u -u\left( \frac{\psi}{\tau} + \frac{1}{(2k\tau)^2} (\psi^2 + \theta^2)\right)  -2 s_x \left( \frac{\psi_x}{\tau} + \frac{2}{(2k\tau)^2} (\psi_x \psi + \psi_y \theta) \right) \\
& -(s + \tau) \left( \frac{\psi_{xx}}{\tau} + \frac{2}{(2k\tau)^2} \left( \psi_{xx} \psi + \psi_{xy} \theta + | \nabla \psi|^2 \right) \right)\\
& - ( 6s(s_x)^2 + 3s^2 u + 6 \tau (s_x)^2 + 6\tau s u) + \eps \Real[ \partial_{xx}( g \rme^{-\rmi(kx +\phi(\varphi))})],
\\
\hat{F}_5 (\xi;\eps, \varphi) =& \Delta v - 2 \tau^2 v -v\left( \frac{\psi}{\tau} + \frac{1}{(2k\tau)^2} (\psi^2 + \theta^2)\right)  - s_x \left( \frac{\theta_x}{\tau} + \frac{2}{(2k\tau)^2} (\theta_x \psi + \theta_y \theta) \right) \\
& - s_y \left( \frac{\psi_x}{\tau} + \frac{2}{(2k\tau)^2} (\psi_x \psi + \psi_y \theta) \right) -(s + \tau) \left( \frac{\theta_{xx}}{\tau} + \frac{2}{(2k\tau)^2} \left( \theta_{xx} \psi + \psi_{yy} \theta + \theta_x\psi_x+\theta_y\psi_y \right)  \right)\\
& - ( 6ss_xs_y + 3s^2 v + 6 \tau s_x s_y + 6\tau s v) + \eps\Real[ \partial_{xy}( g \rme^{-\rmi(kx +\phi(\varphi))})],
\\
\hat{F}_6 (\xi;\eps, \varphi) =& \Delta w - 2 \tau^2 w -w\left( \frac{\psi}{\tau} + \frac{1}{(2k\tau)^2} (\psi^2 + \theta^2)\right)  -2 s_y \left( \frac{\theta_x}{\tau} + \frac{2}{(2k\tau)^2} (\theta_x \psi + \theta_y \theta) \right) \\
& -(s + \tau) \left( \frac{\psi_{yy}}{\tau} + \frac{2}{(2k\tau)^2} \left( \theta_{xy} \psi + \theta_{xy} \theta + | \nabla \theta|^2 \right) \right)\\
& - ( 6s(s_y)^2 + 3s^2 w + 6 \tau (s_y)^2 + 6\tau s w) +\eps \Real[ \partial_{yy}( g \rme^{-\rmi(kx +\phi(\varphi))})].
\end{align*}

Here, $\varphi \in \bbR$, $\tau = \sqrt{1 - k^2}$ , $k^2< \dfrac{1}{3}$, and the variable $\xi = ( s, \psi, \theta, u,v,w)$ is given by the formulas in (\ref{changevariables}), so that we can actually consider $\hat{F}$ as an operator on $\mmD\times \bbR$ for fixed $\eps,\varphi$.

 Since $(2k\tau) \nabla \phi = \langle \psi, \theta \rangle$ we define $\phi$ by
\begin{equation}\label{defphi}  \phi(x,y; \eps, \varphi) = \phi_\mathrm{bd}+ \phi_{\log}, \end{equation}
where
\begin{align*}
 \phi_\mathrm{bd}(x,y;\eps,\varphi) = &\varphi + \frac{1}{2k\tau}\int_{t=0}^1 \left(\hat{\psi}(tx,ty;\eps)x+  \hat{\theta} (tx,ty;\eps)y\right) \\
 \phi_{\log}(x,y;\eps,\varphi) =& \frac{1}{2k\tau} \int_{t=0}^1\left( P_2(tx,ty)x+  P_3 (tx,ty)y\right) =  \frac{1}{2k\tau} \chi \log(\alpha x^2+y^2).
 \end{align*}
 The following  lemma shows that $\phi_\mathrm{bd}$ is a well-defined function.
\begin{lem}\label{phasebounded}
If $\hat{\psi}, \hat{\theta} \in M^{2,2}_{\gamma}$ then for fixed $\eps$ and $\varphi$, the function $\phi_\mathrm{bd}(x,y; \eps, \varphi)$  is well defined, bounded, continuous, and approaches a constant  $\varphi + \Phi_{\infty}(\eps)$  as $| {\bf x} |\rightarrow \infty$.
\end{lem}
\begin{proof} Note that $\phi$ is continuous since $\hat{\psi}, \hat{\theta} \in M^{2,p}_{\gamma} \subset BC^0$.
Lemma \ref{decay} guarantees that for large $|{\bf x} |$ we have $|\hat{\psi}|, |\hat{\theta}| \leq C | {\bf x} | ^{-\gamma-1}$. Therefore, the integrals converge  as $| {\bf x} |\to \infty $.
\end{proof}

The next four lemmas will help us show that the operator $\hat{F}_{\eps, \varphi}: \mmD \times \bbR^3 \rightarrow \mmR$ is well defined.

\begin{lem}\label{littleimbedding}
There exists $C>0$ so that for all $u \in L^p_{\gamma}$ with $ Du \in L^p_{\gamma +1}$, $\langle {\bf x} \rangle^{\gamma}u  \in W^{1,p}$ and, 
\[
\|\langle {\bf x} \rangle^{\gamma}u\|_{W^{1,p}}\leq C \|u\|_{L^p_{\gamma}}+\|Du\|_{L^p_{\gamma +1}}.
\]
\end{lem}
 
 \begin{proof} We need to show that $D(  \langle {\bf x} \rangle^{\gamma} u) \in L^p$. We compute
 \[ D( \langle {\bf x} \rangle^{\gamma} u ) =   D u \cdot \langle {\bf x} \rangle^{\gamma}+ D \langle {\bf x} \rangle^{\gamma} \cdot u= Du\cdot  \langle {\bf x} \rangle^{\gamma} +\gamma  u   {\bf x} (1 + |{\bf x}|^2) ^{\frac{\gamma -2}{2}}.\]
 Since $Du \in L^p_{\gamma+1} \subset L^p_{\gamma}$, we conclude that $Du \cdot  \langle {\bf x} \rangle^{\gamma} \in L^p$. Furthermore, since $|{\bf x}|^p \leq (1 +|{\bf x}|^2)^{p/2}$,
\[ |u \cdot D\langle {\bf x} \rangle^{\gamma} | \leq   |u| |{\bf x} |\langle {\bf x} \rangle^{(\gamma-2)}  \leq   |u|  \langle {\bf x} \rangle^{(\gamma-1)} \leq  |u|  \langle {\bf x} \rangle^{\gamma } \in L^p. \]
This implies that $D(\langle {\bf x} \rangle^{\gamma} u) \in L^p$ and we obtain $ \langle {\bf x} \rangle^{\gamma}  u\in W^{1,p}$.
\end{proof}

\begin{lem}\label{bounded}
For $\gamma >0$, we have the continuous embeddings $M^{2,2}_{\gamma} \hookrightarrow W^{2,2}_{\gamma} \hookrightarrow W^{2,2}\hookrightarrow BC^0$.
\end{lem}
\begin{proof}
The first embedding is due to Lemma \ref{littleimbedding}, the second a consequence of $\gamma >0$, and the last a classical Sobolev embedding in dimension 2.
\end{proof}

\begin{lem}\label{Holder3}
For $\gamma>0$ there exists $C>0$ such that for all $f,g $ with $\langle {\bf x} \rangle^{\gamma+1} f, \langle {\bf x} \rangle^{\gamma+1} g \in W^{1,p}$,
\[
\|fg\|_{L^p_{\gamma+2}}\leq C\|\langle {\bf x} \rangle^{\gamma+1} f\|_{W^{1,p}}\|\langle {\bf x} \rangle^{\gamma+1} g\|_{W^{1,p}}.
\]
\end{lem}

\begin{proof} By Cauchy-Schwartz,
\[
\|fg\langle {\bf x} \rangle^{(\gamma+2)}\|_{L^p}\leq \|f\langle {\bf x} \rangle^{(1+(\gamma/2))}\|_{L^{2p}} \|g\langle {\bf x} \rangle^{(1+(\gamma/2))}\|_{L^{2p}}
\]
which proves the lemma using $\gamma>0$ and the Sobolev embedding $W^{1,p}\hookrightarrow L^{2p}$, in $n=2$.
\end{proof}

Fredholm properties of the Laplacian imply in particular the following more basic estimate  \cite{nirenberg1973null}.
\begin{lem}\cite[Theorem 3.1]{nirenberg1973null}\label{nirenberg}
 If $u \in L^p_{\gamma}$ and $\Delta u \in L^p_{\gamma+2}$ then $u \in M^{2,p}_{\gamma}$ and there exists a constant $C$ such that
 \[  \| u \|_{M^{2,p}_{\gamma}} \leq C \left( \| u\|_{L^p_{\gamma}} + \| \Delta u\|_{L^p_{\gamma+2}}   \right)  .\]
\end{lem}

\begin{rem}\label{betters}
Notice that $u,w\in L^2_{\gamma+2}$ so that $ \Delta s =u+w\in L^2_{\gamma+2}$, by the above lemma  we have that, within the closed subset $\mmD\subset \mmX$, $s \in M^{2,2}_{\gamma}$ with uniform bounds in terms of $s,u,w$.
\end{rem}

To proof Theorem \ref{MainTh}, the following three lemmas establish that each component of the operator $\hat{F}_{\eps, \varphi} : \mmD \times \bbR^3 \rightarrow \mmR$ is well defined. Throughout, we use the standing assumptions  $0<\gamma<1$ and  $g \in W^{2,2}_{\beta}$, with $\beta >\gamma +2$.
\begin{lem}
The component  $\hat{F}_1 : \mmD \times \bbR^3 \rightarrow L^2_{\gamma}$ is well defined. 
\end{lem}

\begin{proof} We can rewrite
\[\hat{F}_1 (\xi;\eps, \varphi) = \Delta s - 2\tau^2 s - \psi  -  \frac{s \psi}{\tau}  - (s + \tau) (  \frac{1}{(2k\tau)^2} (\psi^2 + \theta^2)) - (s^3 + 3s^2 \tau) + \eps \Real [ g\rme^{-\rmi( kx +  \phi(\varphi))}]. \]

A short calculation shows that $\Delta s - 2\tau^2 s - \psi = ( \Delta - 2\tau^2) \hat{s} -\hat{\psi} - c \Delta  P_1$ is the first component of $\hat{T}$, thus well defined. Consider next the term
$ s \psi = ( \hat{s} + c  P_1 ) (\hat{\psi} + c P_2 ). $

Notice that $\hat{s}\psi$ is in $L^2_{\gamma}$ since both $ P_2$ and $\hat{\psi}$, are bounded by Lemma \ref{bounded} and $\hat{s} \in  L^2_{\gamma}$. Since $\hat{\psi} \in L^2_{\gamma}$ and since $ P_1$ is bounded, $\hat{\psi} P_1\in L^2_{\gamma}$ as well. Notice also  that  the product $ P_1\cdot P_2$ is bounded in compact sets and behaves like $ \dfrac{1}{|{\bf x}|^2} $ for large values of $|{\bf x}|$, hence it belongs to $L^2_{\gamma}$ provided $\gamma<1$. This shows the term $ s \psi $ is in the desired space.

Using similar arguments, it is easy to check  that the functions $\psi^2, \theta^2, s^3,s^2$ and $s( \psi^2+ \theta^2)$ are in $L^2_\gamma$. Finally, by Lemma \ref{phasebounded} we know $\phi$ is a bounded continuous function so that we can conclude $\rme^{-\rmi(kx + \phi(\varphi))}$ is a well defined function  in $L^{\infty}$. This implies that the term $\Real [ g\rme^{-\rmi(kx + \phi(\varphi) )}]$ is in $L^2_{\gamma}$ since $g \in L^2_{\beta} \subset L^2_{\gamma}$.
\end{proof}

\begin{lem}\label{welldefinedF1}
The component   $\hat{F}_2 : \mmD \times \bbR^3 \rightarrow L^2_{\gamma+2}$ is well defined.
\end{lem}

\begin{proof} Since we are trying to find solutions near $\xi=0$ we can assume $s/\tau$ is close to zero. We can therefore write
\begin{align*}
 \Delta \psi + (2k)^2 \frac{\tau u}{s +\tau} &= \Delta \psi + (2k)^2 \left(u+\frac{su}{\tau+su}\right)\\
 & = \Delta \hat{\psi} + (2k)^2( \hat{u} + c  \Delta P_2 + (2k)^2 \partial_{xx} P_2 ) + (2k)^2\frac{(\hat{u} + c \partial_{xx}  P_1))(\hat{s}+cP_1)}{\tau+(\hat{u} + c \partial_{xx}  P_1))(\hat{s}+cP_1)}.
 \end{align*}
Notice that the terms $ \Delta \hat{\psi} + (2k)^2 \hat{u} + c [ \Delta P_2 + (2k)^2 \partial_{xx}P_2 ], $ 
represent the second component of the linear operator $\hat{T}: \mmD \times \bbR \rightarrow \mmR$, hence  are well defined. 
It is now straightforward to see that the remaining nonlinear terms are contained in $L^2_{\gamma+2}$. In terms of localization, the most dangerous term is $P_1\partial_{xx}P_1$, which can be bounded as 
\[ \int \left | P_1  \partial_{xx}  P_1 \right|^2 \langle {\bf x} \rangle^{2(\gamma+2)} \leq \int \left| \frac{1}{r^4} \right|^2 r^{2(\gamma+2)} r \rmd r < \infty,\]
since $\gamma<1$.

We  next treat the remaining nonlinearities. Since $s \in L^{\infty}$, we only need to show that the numerators in $\hat{F}_2$ are in $L^2_{\gamma+2}$. It is not hard to mimic the above arguments to show that the terms $u \psi$ and $v \theta$ are in $L^2_{\gamma+2}$, so we will treat the term $s_x \psi_x$ first.  Using the formulas in (\ref{changevariables}) we see that
\[ s_x \psi_x = ( \hat{s}_x + c \partial_{x} P_1) ( \hat{\psi}_x+ c \partial_x  P_2 ).\]

By Lemma \ref{nirenberg} and Remark \ref{betters} we know that $\hat{s}\in M^{2,2}_{\gamma}$. Therefore,  $\hat{s}_x,\hat{\psi}_x\in W^{1,2}_{\gamma+1}$ and we can apply Lemma \ref{Holder3}  to conclude that $\hat{s}_x \hat{\psi}_x\in L^2_{\gamma+2}$. The remaining terms are easily seen to be in the correct space.

Similar arguments show that the functions $(s_x)^2, s_y \psi_y, s_xs_y$ are in the correct space and that, since $ \psi$ and $\theta$ are bounded by Lemma \ref{bounded}, $s_x( s_x \psi+ s_y \theta)\in L^2_{\gamma+2}$. 

Finally, because we are assuming that $g$ is in the space $W^{2,2}_{\beta}$, with $\beta > \gamma+2$, $s_x \in L^2_{\gamma+1}$, and because the terms $\psi$ and $\rme^{-\rmi(kx +  \phi(\varphi))}$ are bounded,
\[-\frac{\eps s_x \Imag[g \rme^{-\rmi(kx +\phi(\varphi))}] }{(s+ \tau)^2} +  \frac{\eps \Imag[ \partial_x( g \rme^{-\rmi(kx +\phi(\varphi))})]}{s+ \tau} \in L^2_{\gamma+2}.\]
 Here, we used  the fact that  $\psi = (2 k \tau) \phi$ so that  $ \partial_x  ( g \rme^{-\rmi(kx +\phi(\varphi))})= [ g_x -i g( k + \psi/(2k\tau)) ] \rme^{-\rmi(kx + \phi(\varphi))} $.
\end{proof}

\begin{lem}
The component $\hat{F}_3 : \mmD \times \bbR^3 \rightarrow L^2_{\gamma+2}$ is well defined.
\end{lem}

\begin{proof} The proof is almost identical to the proof of Lemma \ref{welldefinedF1} and is omitted here.
\end{proof}

Finally, we show

\begin{lem}
The component $\hat{F}_4 : \mmD\times\bbR^3 \rightarrow W^{-2,2}_{\gamma+2}$ is well defined. Moreover, the nonlinear part of $\hat{F}_4$ actually belongs to $L^2_{\gamma+2}$.
\end{lem}

\begin{proof} We can rewrite $\hat{F_4}$ as
\[ \Delta u - 2 \tau^2 u - \psi_{xx} - \frac{ s \psi_{xx}}{\tau} -u\left( \frac{\psi}{\tau} + \frac{1}{(2k\tau)^2} (\psi^2 + \theta^2)\right)  -2 s_x \left( \frac{\psi_x}{\tau} + \frac{2}{(2k\tau)^2} (\psi_x \psi + \psi_y \theta) \right) \]
\[ -(s + \tau) \left(   \frac{2}{(2k\tau)^2} \left( \psi_{xx} \psi + \psi_{xy} \theta + | \nabla \psi|^2 \right) \right)  - ( 6s(s_x)^2 + 3s^2 u + 6 \tau (s_x)^2 + 6\tau s u) + \eps \Real[ \partial_{xx}( g \rme^{-\rmi(kx +\phi(\varphi))})].\]

Notice that
\[ \Delta u - 2 \tau^2 u - \psi_{xx} = \Delta\hat{u} - 2 \tau^2 \hat{u} - \hat{\psi}_{xx} + c \Delta( \partial_{xx} P_1), \]
 is just the fourth component of the linear operator $\hat{T}$, thus well defined. Furthermore, notice that $\hat{\psi}_{xx} \in L^2_{\gamma+2}$. Since $\Delta \partial_{xx} P_1$ behaves like $\dfrac{1}{|{\bf x}|^5}$ for large $|{\bf x}|$ and is bounded in compact sets,  these two term now also belong to $L^2_{\gamma+2}$.
The arguments used to show that the remaining nonlinearites are in the space $L^2_{\gamma+2}$ are the same as the once used in the above lemmas, we will omit the details here.
\end{proof}

Having shown the result for the operator $\hat{F}_4$ it is not hard to see that the operators $\hat{F}_5,\hat{F}_6 : \mmD\times \bbR^3 \rightarrow W^{-2,2}_{\gamma+2}$ are well defined.

\begin{rem}\label{betteru}
Since all the nonlinear terms are in $L^2_{\gamma+2}$, including $\hat{\psi}_{xx}$ and $\Delta \partial_{xx} P_1$, then $\Delta\hat{u} - (2\tau^2)\hat{u} \in L^2_{\gamma+2}$. This implies that for the solution, $\hat{u} \in W^{2,2}_{\gamma+2}$. The same observation holds for  $\hat{v}$ and $\hat{w}$.
\end{rem}

In the next lemma we show that there exist a neighborhood $\mmU$ of $\xi \in \mmD \times \bbR$ such that the operator $\hat{F}_{\eps, \varphi}: \mmU \times \bbR^3 \rightarrow \mmR$ is smooth.

\begin{lem}\label{differentiable}
Let  $0< \gamma<1$ and  $g \in W^{2,2}_{\beta}$, with $\beta> \gamma+2$. Then the operator $\hat{F}_{\eps, \varphi} : \mmD \times\bbR^3 \rightarrow \mmR$ is of class $C^\infty$ in a neighborhood the origin.
\end{lem}

\begin{proof} Most nonlinear terms are defined via superposition (or Nemytskii) operators, via smooth algebraic functions, that are automatically smooth once well defined. We therefore concentrate on the term $g \rme^{-\rmi\phi}$ and its derivatives. Recall that
\[  \phi(x,y; \eps, \varphi) = \phi_\mathrm{bd}+ \phi_{\log}, \]
where
\begin{align*}
 \phi_\mathrm{bd}(x,y;\eps,\varphi) = &\varphi + \frac{1}{2k\tau}\int_{t=0}^1 \left(\hat{\psi}(tx,ty;\eps)x+  \hat{\theta} (tx,ty;\eps)y\right) \\
 \phi_{\log}(x,y;\eps,\varphi) =& \frac{1}{2k\tau} \int_{t=0}^1\left( P_2(tx,ty)x+  P_3 (tx,ty)y\right) =  \frac{1}{2k\tau} c \chi \log(\alpha x^2+y^2).
\end{align*}
In order to show smoothness, we factor
\[
g \rme^{-\rmi\Phi}=\left(g\langle\mathbf{x}^{\gamma+2-\beta}\rangle\right)\cdot\rme^{-\rmi\Phi_\mathrm{bd}} \cdot \left(\langle\mathbf{x}^{\beta-\gamma-2}\rangle \rme^{-\rmi\Phi_\mathrm{log}}\right)=:G_1\cdot G_2 \cdot G_3.
\]
Clearly, $G_1\in L^2_{\gamma+2}$. By Lemma \ref{phasebounded}, $\int\psi,\int\theta\in L^\infty$, so that $G_2\in L^\infty$ is bounded as a superposition operator. It remains to show that $G_3$ is differentiable with values in $L^\infty$. This can be readily established, showing that the derivative with respect to $c$ is 
\[
\partial_cG_3=\langle\mathbf{x}^{\beta-\gamma-2}\rangle  \rme^{-\rmi\Phi_\mathrm{log}}\chi \log(\alpha x^2+y^2),
\]
hence bounded in $L^\infty$. Higher derivatives are bounded for the same reasons, which establishes the claim.
\end{proof}

\section{Expansions and proof of main result}\label{Mainresult}

In this last subsection we use Theorem \ref{MainTh} to proof Theorem \ref{MainTh1}  and derive the expansions for the stationary solutions to the perturbed Ginzburg-Landau equation near roll patterns. 
\begin{proof}[of Theorem \ref{MainTh1}] Recall the  Ansatz
\[ A(x,y; \eps, \varphi) = (\sqrt{1 -k^2} + s(x,y;\eps, \varphi)+ c(\eps, \varphi) P_1(x,y)) \rme^{\rmi kx + \rmi \phi(x,y;\eps, \varphi)+ \rmi \frac{c(\eps, \varphi)}{2k \sqrt{1-k^2}}P_2(x,y)} \]
where $\Phi$ was defined in \eqref{defphi} and $P_1,P_2$ in \eqref{e:p1p2}.
%
From Theorem \ref{MainTh} we know there exists  a neighborhood, $\mmU$, of $\mmD \times \bbR$ where the operator $\hat{F}_{\eps,\varphi}$ is continuously differentiable with invertible derivative at the origin, $\eps=0$. The Implicit Function Theorem therefore guarantees the existence of solutions $\xi(\eps, \varphi) $ near $\xi(0, \varphi)=0$. In particular, we know that $s \in W^{2,2}_{\gamma}$, and $\psi, \theta \in M^{2,2}_{\gamma}$.

We define 
\[ S(x,y; \eps, \varphi) = \left(\sqrt{1-k^2} + s(x,y;\eps, \varphi) + c(\eps, \varphi)  P_1(x,y) \right)\]
and
\[ \Phi( x,y;\eps, \varphi) =kx + \phi. \]

Since $s(x,y;\eps, \varphi) \in W^{2,2}_{\gamma}$, Lemma \ref{bounded}  ensures that if $s(x,y; \eps, \varphi) \sim \rmO(\langle {\bf x} \rangle^{-\gamma})$.  Also,  by definition, $P_1(x,y) \sim \rmO(\langle {\bf x} \rangle^{-1})$, and
\[ \lim_{\mathrm{x}\to\infty} S(x,y;\eps, \varphi) =S_{\infty}= \sqrt{1-k^2}.\]

By Lemma \ref{phasebounded}, $\phi_\mathrm{bd} \to\varphi + \Phi_{\infty}(\epsilon)$ for $\mathbf{x}\to\infty$ so that 
\[ \Phi(x,y;\eps,\varphi)-kx-\frac{c(\eps,\varphi)}{2k \sqrt{1-k^2}}\log(\alpha x^2+y^2)\to  \Phi_\infty(\eps)+\varphi,\]
as $|{\bf x}| \to \infty$.

To find an expression for $c(\eps, \varphi)$ we expand $\xi = \eps \hat{\xi} + \rmo(\eps)$. Gathering terms of order $\eps$ results in the  system $\hat{T} \hat{\xi} = \hat{f}$. Inspecting the second component of this system, we find
\[ \Delta \hat{\psi} + b \partial_{xx} \hat{u} + \hat{c} \left[\Delta P_2 + b \partial_{xx} P_1  \right] = \dfrac{1}{\sqrt{1-k^2}}\Imag[	 ( g_x-ikg) \rme^{-\rmi(kx+\varphi)}].\]
Taking the scalar product with $x$ and solving for $\hat{c}$, we obtain after integration by parts in $x$,
\[ \hat{c} =\frac{\sqrt{1-3k^2}}{\pi(1-k^2)} \iint \Imag[ g\rme^{-\rmi( kx+ \varphi)}].
\]
Hence $c( \eps, \varphi) = \eps c_1( \varphi) + \rmo(\eps)$ with $c_1(\varphi) = \hat{c}$.

\end{proof}

\section{Appendix}

\begin{lem}\label{decay}
If $f \in M^{2,2}_{\gamma}$ then $|f| \leq  {C} \| f\|_{M^{2,2}_{\gamma}}  \langle {\bf x} \rangle^{-\gamma-1}$ as $| {\bf x} |\rightarrow \infty$.
\end{lem}
\begin{proof} Since $M^{2,2}_{\gamma}$ is the completion of $C^{\infty}_0$ under the norm $\| \cdot\|_{M^{2,2}_{\gamma}}$, it suffices to show that the result holds for $f \in C^{\infty}_0$. In polar coordinates, we have, up to constants 
  \begin{align*}
  \int |f(\theta, R)|^2 \rmd\theta &\leq \int \left( \int_{\infty}^R |f_r(\theta, s)| \rmd s \right)^2  \rmd\theta
  =\int \left( \int_{\infty}^R s^{-\gamma -3/2} s^{\gamma+1} | f_r(\theta, s)|  s^{1/2} \rmd s \right)^2  \rmd\theta\\
  & \leq \int \left(\int_{\infty}^R s^{-2(\gamma + 3/2)} \rmd s \right) \left( \int_{\infty}^R s^{2(\gamma+1)} |f_r(\theta,s)|^2 s\rmd s \right)  \rmd\theta \\
  &\lesssim R^{-2(\gamma+3/2)+1} \int \int_{\infty}^R s^{2( \gamma+1)}|f_r(\theta,s)|^2 s\rmd s\rmd\theta ,
  \end{align*}
which gives
\begin{equation}\label{es1} \| f(\cdot, R)\|_{L^2} \lesssim R^{-\gamma-1}\|
 \nabla f\|_{L^2_{\gamma+1}}.
\end{equation}
Similarly,
\begin{align*}
\int |f_{\theta}(\theta, R) |^2\rmd\theta &\leq \int \left( \int_{\infty}^R |f_{r \theta}(\theta,s)|\rmd s \right)^2 \rmd\theta
= \int \left( \int_{\infty}^R s^{-\gamma -5/2} s^{\gamma+2} |f_{r \theta}(\theta,s)|  s^{1/2}\rmd s \right)^2 \rmd\theta\\
& \leq \int \left( \int_{\infty}^R s^{-2(\gamma+5/2)} \rmd s \right) \left( \int_{\infty}^R  s^{2(\gamma+2)} |f_{r \theta}(\theta,s)|^2  s\rmd s \right) \rmd\theta \\
& \lesssim R^{-2(\gamma+5/2)+1} \int  \int_{\infty}^R  s^{2(\gamma+2)} |f_{r \theta}(\theta,s)|^2  s\rmd s  \rmd\theta.
\end{align*}
This gives
\begin{equation}\label{es2} \| f_\theta(\cdot, R)\|_{L^2}\leq  R^{-\gamma-2} \| f_{r\theta} \|_{L^2_{\gamma+2}}.\end{equation}
Combining \eqref{es1} and \eqref{es2} and using the interpolation inequality \cite[Thm 5.9]{adams}
\[
\|f(\cdot,R)\|^2_\infty\leq \|f(\cdot,R)\|_{L^2} \|f(\cdot,R)\|_{H^1},
\]
now proves the claim.
\end{proof}

%
%
%

 \bibliographystyle{siam}	
\bibliography{kondratiev}	
\end{document}